\documentclass[11pt]{amsart}

\usepackage{import}

\newcommand{\tr}{\operatorname{tr}} % Traza
\newcommand{\rk}{\operatorname{rk}} % Rango
\renewcommand{\div}{\operatorname{div}} % Divergencia
\newcommand{\prodesc}[2]{\langle #1,#2 \rangle} % Producto escalar
\newcommand{\End}{\operatorname{End}}

\newcommand{\SV}{\mathrm{SV}}
\newcommand{\GF}{\mathrm{GF}} %
\newcommand{\SSCV}{\mathrm{SSCV}} %
\newcommand{\JV}{\mathrm{JV}} %

\newcommand{\so}{\mathfrak{so}}

\newcommand{\Lcal}{\mathcal{L}}

\newcommand{\Rcal}{\mathcal{R}}

\newcommand{\Gcal}{\mathcal{G}}
\newcommand{\Riccal}{\mathcal{R}\hspace{-.75pt}\mathrm{ic}}
\newcommand{\Cinf}{C^\infty}

\newcommand{\der}[2]{\frac{d #1}{d #2}} % Derivada total de #1 respecto a #2
\calclayout

\newtheorem{theorem}{Theorem}

\newtheorem{lemma}[theorem]{Lemma}
\newtheorem{proposition}[theorem]{Proposition}

\theoremstyle{definition}
\newtheorem{definition}[theorem]{Definition}
\theoremstyle{remark}

\usepackage{enumerate}

\usepackage{mathrsfs}
\usepackage[
	scr=boondox,
	cal=dutchcal %euler %stix2-plain %dutchcal
]{mathalpha}

\usepackage{xcolor}
\definecolor{color-citas}{RGB}{9,21,128}

\usepackage{hyperref}
\hypersetup{
    colorlinks=true,
    linkcolor=color-citas,
    citecolor=color-citas
}

\usepackage{cleveref}

\begin{document} 

% \maxtocdepth{subsection} % Profundidad del table of contents
\renewcommand{\bibname}{References} % Nombre de las referencias

\title[On the equivalence of generalized Ricci curvatures]{On the equivalence of\\  generalized Ricci curvatures}
\author{Gil R. Cavalcanti}
\author{Jaime Pedregal}
\address{Department of Mathematics, Utrecht University, 3508 TA Utrecht, The Netherlands}
\email{g.r.cavalcanti@uu.nl, j.pedregalpastor@uu.nl}

\author{Roberto Rubio}
\address{Universitat Aut\`onoma de Barcelona, 08193 Barcelona, Spain}
\email{roberto.rubio@uab.es}

%\author{Gil R. Cavalcanti, Jaime Pedregal and Roberto Rubio}
%\date{2024}

\thanks{This project has been supported by FEDER/AEI/MICINN through the grant PID2022-137667NA-I00 (GENTLE). G.C. and J.P. are also supported by the Open Competition grant number OCENW.M.22.264 from NWO, and R.R. by AGAUR under 2021-SGR-01015 and by FEDER/AEI/MICINN under RYC2020-030114-I}

\begin{abstract}
We prove the equivalence between the several notions of generalized Ricci curvature found in the literature. As an application, we characterize when the total generalized Ricci tensor is symmetric.

% As an application, we characterize when the total generalized Ricci tensor is symmetric.
\end{abstract}

\maketitle
\section{Introduction}

In the context of generalized geometry and Courant algebroids, the notion of generalized Ricci curvature has found applications, among others, in supergravity \cite{coimbra,garcia-fernandez-heterotic}, string effective actions \cite{jurco-vysoky,jurco-moucka-vysoky} and generalized Ricci flow \cite{pulmann-severa-youmans,garcia-fernandez-streets,streets-strickland-constable-valach}.

However, the definitions used differ and it was drawn to our attention by García-Fernández that, apart from the case of string algebroids (which the authors of \cite{garcia-fernandez-gonzalez-molina-streets} were aware of), no general proof of their equivalence is known. The purpose of this note is to establish this equivalence and also show how it can be applied. % , which also motivated a new approach to total curvatures.

%%
%\section{Generalized Ricci curvatures}

\bigskip

We recall that a Courant algebroid $(E\to M,\rho,\prodesc{\cdot}{\cdot},[\cdot,\cdot])$ consists of a vector bundle $E$ (whose sections will be denoted by $a,b,c,e\in\Gamma(E)$), a symmetric bilinear nondegenerate 2-form  $\prodesc{\cdot}{\cdot}$ on $E$, an anchor map (vector bundle morphism) $\rho:E\to TM$ and a bilinear bracket $[\cdot,\cdot]$ on $\Gamma(E)$ such that
\begin{enumerate}[1)]
	\item $[a,[b,c]]=[[a,b],c]+[b,[a,c]]$,
	\item $\Lcal_{\rho a}\prodesc{b}{c}=\prodesc{[a,b]}{c}+\prodesc{b}{[a,c]}$,
	\item $2[a,a]=\rho^*d\prodesc{a}{a}$, where we identify $E^*\cong E$ via $\prodesc{\cdot}{\cdot}$.
\end{enumerate}
As a consequence, $[a,fb]=f[a,b]+(\Lcal_{\rho a}f)b$, where $f\in\Cinf(M)$ from now on.

% We will use the notation $\Omega^k(E,W):=\Gamma(\Lambda^kE^*\otimes W)$, for any vector bundle $W\to M$.

\begin{definition}%[\cite{gualtieri}]
A \textbf{ generalized connection} on $E$ is a linear operator $D:\Gamma(E)\to\Gamma(E^*\otimes E)$ such that
\[
D(fa)=fDa+\rho^*df\otimes a,\quad\rho^*d\prodesc{a}{b}=\prodesc{Da}{b}+\prodesc{a}{Db}.
\]
We shall use the notation $D_b a:=(Da)b$. % CHECK
\end{definition}

The na\"ive curvature is the linear operator $\Rcal_0:\Gamma(E)^2\to\Gamma(\so(E))$ given by
\begin{equation}\label{eq:R-naive}
\Rcal_0(a,b)c:=D_aD_bc-D_bD_ac-D_{[a,b]}c.     
\end{equation}
From the very start of the theory of generalized connections it was observed that this curvature is not fully tensorial: it is tensorial in the second component, but not in the first one:
\[
\Rcal_0(fa,b)=f\Rcal_0(a,b)+\prodesc{a}{b}D_{\rho^*df}.
\]
%(note that $D_{\rho^*df}\in\Gamma(\End E)$, since $\rho\circ\rho^*=0$). 

It was later understood that it is actually a Keller--Waldmann form, in the sense of \cite{keller-waldmann,cueca-mehta}. Nevertheless, there have been different attempts to make it tensorial, leading to different notions of generalized curvature.%, each with its advantages and disadvantages.

% We will write $\Rcal_D$ whenever we want to emphasize the dependence on $D$. 
% Note that if $W$ is equipped with a metric and $D$ is compatible with it then $\Rcal$ lands in sections of
% \[
% \so(W):=\{B\in\End W: \prodescp{Bw}{u}+\prodescp{w}{Bu}=0,\text{ for $w,u\in W$}\}.
% \]

The first approach to make it tensorial consists of letting $a$ and $b$ take values in orthogonal bundles:  either both in an isotropic subbundle of $E$ \cite{gualtieri} or each one in a different orthogonal subbundle of a generalized metric \cite{garcia-fernandez-heterotic,garcia-fernandez-spinors}.

\begin{definition}
A \textbf{generalized metric} on $E$ is a subbundle $V_+\subseteq E$ such that $\prodesc{\cdot}{\cdot}_{|V_+}$ is nondegenerate.
\end{definition}
We then have that $V_-:=V_+^\perp$ is also a generalized metric and that $E=V_+\oplus V_-$. 
A generalized metric is equivalently given by an endomorphism $\Gcal:E\to E$ such that $\Gcal^2=1$ and $\Gcal^*=\Gcal$, where $^*$ denotes the adjoint with respect to $\prodesc{\cdot}{\cdot}$ (in this case, $V_\pm=\ker(\Gcal\mp 1)$). For $a\in E$ we will denote by $a_\pm:=\frac{1}{2}(1\pm\Gcal)a\in V_\pm$ the projections onto $V_\pm$.

\begin{definition}[\cite{garcia-fernandez-heterotic}]
The curvature tensor fields $\Rcal^\pm_\GF\in\Gamma(V_\pm^*\otimes V_\mp^*\otimes\so(E))$ are given by
\[
\Rcal^\pm_\GF(a_\pm,b_\mp):=\Rcal_0(a_\pm,b_\mp).\qedhere
\]
\end{definition}
When $D$ is compatible with $\Gcal$ (to be defined below), the tensors $\Rcal^\pm_\GF$ satisfy two of the four symmetries of the Riemann tensor \cite[Prop. 4.1]{garcia-fernandez-spinors}, as the other two are not even well defined. 

% if $D$ is torsion-free (to be defined below) then it also satisfies an algebraic Bianchi identity \cite[Prop. 4.1]{garcia-fernandez-spinors}. Out of the four classical symmetries of the Riemann tensor, $\Rcal^\pm_\GF$ satisfies two, because the other two are not even well defined.

An alternative approach to tensoriality is to tweak \eqref{eq:R-naive} in such a way that it also satisfies all four symmetries of the classical Riemann tensor \cite[Prop. 4.10, Thm. 4.13]{jurco-vysoky}.

\begin{definition}[\cite{jurco-vysoky}]
The curvature tensor $\Rcal_\JV\in\Gamma(E^*\otimes E^*\otimes\so(E))$ is given by
\[
\prodesc{\Rcal_\JV(a,b)c}{e}:=\frac{1}{2}(\prodesc{\Rcal_0(a,b)c}{e}+\prodesc{\Rcal_0(c,e)a}{b}+\prodesc{(Da)^*b}{(Dc)^*e}).
\] 
\end{definition}

To any of these curvatures we can associate a Ricci tensor. This is a concept that plays a central role in generalized geometry and, according to \cite{streets-strickland-constable-valach}, can be traced back, although not in a generalized-geometric way, to \cite{siegel}. %, although not in the language of generalized geometry. % ,  % a

The aim of this note is to prove that the Ricci tensors induced by the two curvatures above, together with yet a third definition of the Ricci tensor, are equivalent, Theorems \ref{thm:ricci} and \ref{thm:ricci-2}. As an application, in Proposition \ref{prop:symmetric-ricci}, we characterize when the total generalized Ricci tensor is symmetric, answering a question raised in \cite{garcia-fernandez-spinors}.

% This curvature is indeed tensorial and it satisfies the classical symmetries of the Riemann tensor.

\section{Generalized Ricci curvatures of a connection}

%Attached to these curvature tensors, there are generalized Ricci curvature tensors. 
For $\Rcal^\pm_{\GF}$, we consider \textbf{metric} connections, those $D$ such that $D_a(\Gamma(V_\pm))\subseteq\Gamma(V_\pm)$ for all $a\in\Gamma(E)$, or, equivalently, $D\Gcal=0$. Note that in this case,
\[
\Rcal_\GF^\pm\in\Gamma(V_\pm^*\otimes V_\mp^*\otimes\so(V_\pm))\oplus\Gamma(V_\pm^*\otimes V_\mp^*\otimes \so(V_\mp)).
\]
Only the first piece is relevant, since the other one can be recovered from $\Rcal_\GF^\mp$. Indeed, it is immediate to see that
\[
\prodesc{\Rcal_\GF^\pm(a_\pm,b_\mp)c_\mp}{e_\mp}=-\prodesc{\Rcal_\GF^\mp(b_\mp,a_\pm)c_\mp}{e_\mp}.
\]
To avoid redundant information, from now on we will consider only the first piece, i.e.,
\[
\Rcal_\GF^\pm\in\Gamma(V_\pm^*\otimes V_\mp^*\otimes\so(V_\pm)).
\]

\begin{definition}[\cite{garcia-fernandez-spinors}]
Let $\Gcal$ be a generalized metric on $E$ and $D$ a metric connection. The tensor $\Riccal_{\GF}^\pm\in\Gamma(V_\mp^*\otimes V_\pm^*)$ is given by
\[
\Riccal_{\GF}^\pm(a_\mp,b_\pm):=\tr_\pm(\Rcal^\pm_{\GF}(\cdot,a_\mp)b_\pm),
\]
where $\tr_\pm$ is the trace on $V_\pm$.
\end{definition}

For $\Rcal_\JV$ we can give a direct definition.

\begin{definition}[\cite{jurco-vysoky}] The tensor $\Riccal_\JV\in\Gamma(E^*\otimes E^*)$ is given by
\[
\Riccal_\JV(a,b):=\tr(\Rcal_\JV(\cdot,a)b).
\]
\end{definition}

And, in the presence of a metric $\Gcal$, the following tensor has also been considered recently. %  but not necessarily a metric connection, 

\begin{definition}[\cite{streets-strickland-constable-valach}]
   The tensor $\Riccal_{\SSCV}\in\Gamma(E^*\otimes E^*)$ is given by
\[
\Riccal_\SSCV(a,b):=\Riccal_\JV(a,b)-\Riccal_\JV(\Gcal a,\Gcal b).
\]
\end{definition}
Note that for $a_\pm,b_\pm\in V_\pm$ we have that
\[
\Riccal_\SSCV(a_\pm,b_\pm)=0\qquad\text{and}\qquad \Riccal_\SSCV(a_\pm,b_\mp)=2\Riccal_\JV(a_\pm,b_\mp).
\]

We show now the equivalence of these three tensors:

\begin{theorem} \label{thm:ricci}
For $\Gcal$ a generalized metric and $D$ a metric connection on $E$, given any $a_\pm,b_\pm\in\Gamma(V_\pm)$,
	\[ \Riccal_\SSCV(a_\mp,b_\pm)=2\Riccal_\JV(a_\mp,b_\pm)=\Riccal^\pm_\GF(a_\mp,b_\pm)+\Riccal^\mp_\GF(b_\pm,a_\mp).
	\]
\end{theorem}

\begin{proof}
Since $(Da_\pm)^*b_\mp=0$ and $\prodesc{\Rcal(c_\pm,e_\pm)a_\pm}{b_\mp}=0$ for $c_\pm,e_\pm\in\Gamma(V_\pm)$, both of which follow from $D\Gcal=0$, we have that
\[
\prodesc{\Rcal_\JV(a_\pm,b_\mp)c_\pm}{e_\pm}=\frac{1}{2}\prodesc{\Rcal_\GF^\pm(a_\pm,b_\mp)c_\pm}{e_\pm}.
\]
Note also that, because $[a_\pm,b_\mp]=-[b_\mp,a_\pm]$ and $\Rcal(a_\pm,b_\mp)\in\so(V_\mp)$,
\[
\prodesc{\Rcal(a_\pm,b_\mp)c_\mp}{e_\mp}=\prodesc{\Rcal(b_\mp,a_\pm)e_\mp}{c_\mp},
\]
which gives as well that
\[
\prodesc{\Rcal_\JV(a_\mp,b_\mp)c_\pm}{e_\mp}=\frac{1}{2}\prodesc{\Rcal^\mp_\GF(e_\mp,c_\pm)b_\mp}{a_\mp}.
\]

Let $\{e_i^\pm\}_i$ be an orthonormal frame for $V_\pm$ and let $\{\tilde e_i^\pm\}_i$ be its dual frame, by which we mean that $\prodesc{\tilde e_i^\pm}{e_j^\pm}=\delta_{ij}$ (the range of the indices $i$, $j$ may of course be different for $V_+$ and $V_-$). Then finally we have that
\begin{align*}
2\Riccal_\JV(a_\mp,b_\pm) &= 2\sum_i \prodesc{\Rcal_\JV(e^\pm_i,a_\mp)b_\pm}{\tilde e^\pm_i}+2\sum_j\prodesc{\Rcal_\JV(e^\mp_j,a_\mp)b_\pm}{\tilde e^\mp_j} \\
&=\sum_i\prodesc{\Rcal^\pm_\GF(e^\pm_i,a_\mp)b_\pm}{\tilde e^\pm_i}+\sum_j \prodesc{\Rcal^\mp_\GF(\tilde e^\mp_j,b_\pm)a_\mp}{e_j^\mp} \\
&= \Riccal^\pm_\GF(a_\mp,b_\pm)+\Riccal^\mp_\GF(b_\pm,a_\mp). \qedhere  
\end{align*}
\end{proof}

%%%%

\section{Generalized Ricci curvature of a pair $(\Gcal,\div)$}

The theory of generalized connections becomes more interesting when the dependence on the connection is lost. For this, we recall the divergence operator and the torsion.

\begin{definition}[\cite{alekseev-xu}]
A \textbf{divergence operator} on $E$ is a linear operator $\div:\Gamma(E)\to\Cinf(M)$ such that
\[
\div(fa)=f\div a+\Lcal_{\rho a}f. % } ,\quad\text{ for $a\in\Gamma(E)$ and $f\in\Cinf(M)$
\]
\end{definition}

For a generalized connection $D$ on $E$, we define its \textbf{divergence} as
\[
\div a:=\tr(Da),
\]
which is indeed a divergence operator on $E$.

The \textbf{torsion} of $D$ is the tensor field $T\in\Gamma(\wedge^3 E^*)$ given, as a map $\Gamma(\wedge^2 E^*)\to \Gamma(E)$, by
\[
T(a,b):=D_ab-D_ba-[a,b]+(Da)^*b.
\]
In the presence of a metric, the torsion is said to be of \textbf{pure type} if $T\in\Gamma(\wedge^3 V_+ \oplus \wedge ^3 V_-)$. This is the case \cite[Lem. 3.2]{garcia-fernandez-spinors} if and only if 
\begin{equation}\label{eq:pure-torsion}
D_{a_\mp}b_\pm=[a_\mp,b_\pm]_\pm.    
\end{equation}
Given a pair $(\Gcal,\div)$, metric connections with divergence $\div$ and pure-type torsion always exist, provided that $\rk V_\pm\neq 1$ \cite[Lem. 2.4, Prop. 3.3]{garcia-fernandez-spinors}. Moreover, for a pair $(\Gcal,\div)$, all metric connections with pure-type torsion and divergence $\div$ give the same tensor $\Riccal^\pm_\GF$ \cite[Prop. 4.4]{garcia-fernandez-spinors}, so we can talk of the generalized Ricci tensor of $(\Gcal,\div)$. We give an alternative and more straightforward proof of this independence, since it follows from the trace operators being parallel.

\begin{definition}
The \textbf{trace} operator is the section $\tr\in\Gamma((\End E)^*)$ given on decomposable endomorphisms by
\[
\tr(\alpha\otimes a):=\alpha(a),\quad\text{ for $\alpha\in \Gamma(E^*)$ and $a\in\Gamma(E)$.}
\]
Similarly, we define $\tr_1:E^*\otimes E^*\otimes E\to E^*$ by
\[
\tr_1(\alpha\otimes\beta\otimes a):=\alpha(a)\beta,\quad\text{ for $\alpha,\beta\in\Gamma(E^*)$ and $a\in\Gamma(E)$.}
\]
\end{definition}

Note that if $B\in\Gamma(E^*\otimes \End E)$ and we use the notation $B_ab:=B(a)b$, so that both $B_a$ and $Bb$ are endomorphisms, then $(\tr_1B)(b)=\tr(Bb)$. By construction, generalized connections preserve these traces, by which we mean that $D\tr=0$ and $D\tr_1=0$, or, equivalently, that $D_a$ commutes with $\tr$ and $\tr_1$.

% \begin{lemma} \label{lem:traces-are-parallel}
% For any generalized connection $D$, the traces $\tr$ and $\tr_1$ are parallel. In particular, if $B\in \Gamma(E^*\otimes \End E)$ is such that $\tr_1B=0$, then $\tr_1(D_aB)=0$ for all $a$. % In particular, the bundle $(E^*\otimes\End E)_0$ is parallel, by which we mean that $D_a(\Omega^1_0(E,\End E))\subseteq \Omega^1_0(E,\End E)$ for all $a\in\Gamma(E)$.
% \end{lemma}

% \begin{proof}
% Let $\alpha\in\Gamma(E^*)$ and $a,b\in\Gamma(E)$. Then
% \[
% \begin{aligned}
% (D_a\tr)(\alpha\otimes b) &=\Lcal_{\rho a}(\alpha(b))-\tr(D_a\alpha\otimes b+\alpha\otimes D_ab) \\
% &=\Lcal_{\rho a}(\alpha(b))-D_a\alpha(b)-\alpha(D_ab)=0.
% \end{aligned}
% \]
% That $\tr_1$ is parallel now follows easily. Finally, if $B\in\Gamma(E^*\otimes \End E)$ is such that $\tr_1B=0$, then
% \[
% \tr_1(D_aB)=D_a(\tr_1B)-(D_a\tr_1)B=0. \qedhere
% \]
% % If $B\in\Omega^1(E,\End E)$, then
% % \[
% % \begin{aligned}
% % ((D_a\tr_1)B)(b) &=(D_a(\tr_1B))(b)-(\tr_1(D_aB))(b) \\
% % &= \Lcal_{\rho a}((\tr_1B)(b))-(\tr_1B)(D_ab)-(\tr_1(D_aB))(b) \\
% % &= \Lcal_{\rho a}(\tr(Bb))-\tr(BD_ab)-\tr((D_aB)b) \\
% % &= (D_a\tr)(Bb)=0.
% % \end{aligned}
% % \]
% \end{proof}

\begin{proposition}
Given $(\Gcal,\div)$ and two metric connections $D,D'$ with divergence $\div$ and pure-type torsion, we have that $\Riccal^\pm_{\GF,D}=\Riccal^\pm_{\GF,D'}$.
\end{proposition}

\begin{proof}
Write $D'=D+B$, for $B\in\Gamma(E^*\otimes\so(E))$. Since both $D'$ and $D$ are compatible with $\Gcal$ we have that actually $B=B^++B^-$, with $B^\pm\in\Gamma(E^*\otimes\so(V_\pm))$. Since both connections have the same mixed-type operators,
\[
D'_{a_\pm}b_\mp=D_{a_\pm}b_\mp+B^\mp_{a_\pm}b_\mp=D'_{a_\pm}b_\mp+B^\mp_{a_\pm}b_\mp,
\]
from where it follows that $B^\mp_{a_\pm}b_\mp=0$, i.e., $B^\pm\in\Gamma(V_\pm\otimes\so(V_\pm))$. Lastly, since both $D$ and $D'$ have the same divergence,
\[
\div a_\pm=\tr(D'a_\pm)=\tr(Da_\pm + Ba_\pm)=\div a_\pm+(\tr_1 B^\pm)(a_\pm),
\]
so $\tr_1B^\pm=0$.

Hence, using that $[a_\pm,b_\mp]=-[b_\mp,a_\pm]$,
\begin{align*}
\Rcal^\pm_{\GF,D'}(a_\pm,b_\mp)c_\pm &=D'_{a_\pm}(D_{b_\mp}c_\pm)-D'_{b_\mp}(D_{a_\pm}c_\pm + B^\pm_{a_\pm}c_\pm)-D_{[a_\pm,b_\mp]}c_\pm -B^\pm_{[a_\pm,b_\mp]_\pm}c_\pm \\
&=\Rcal^\pm_{\GF,D}(a_\pm,b_\mp)c_\pm + B^\pm_{a_\pm}(D_{b_\mp}c_\pm)-D_{b_\mp}(B^\pm_{a_\pm}c_\pm)+B^\pm_{D_{b_\mp}a_\pm}c_\pm \\
&=\Rcal^\pm_{\GF,D}(a_\pm,b_\mp)c_\pm-(D_{b_\mp}B^\pm)_{a_\pm}c_\pm.
\end{align*}
This gives finally that, since $\tr_1B^\pm=0$, and $D$ preserves traces,
\[
\Riccal^\pm_{\GF,D'}(b_\mp,c_\pm)=\Riccal^\pm_{\GF,D}(b_\mp,c_\pm)-(\tr_1(D_{b_\mp}B^\pm))(c_\pm)=\Riccal_{\GF,D}^\pm(b_\mp,c_\pm). \qedhere
\]
\end{proof}

This reliance on the pair $(\Gcal,\div)$ motivated yet another notion of Ricci curvature.

\begin{definition}[\cite{severa-valach}]
For $(\Gcal,\div)$, the tensor $\Riccal^\pm_{\SV}\in\Gamma(V_\mp^*\otimes V_\pm^*)$ is given by
\[
\Riccal^\pm_\SV(a_\mp,b_\pm):= \div([a_\mp,b_\pm]_\pm)-\Lcal_{\rho a_\mp}(\div b_\pm)-\tr_\pm[[\cdot,a_\mp]_\mp,b_\pm]_\pm.
\]
\end{definition}

We can now show our second main result.

\begin{theorem} \label{thm:ricci-2}
For a pair $(\Gcal,\div)$ on $E$ with $\rk V_\pm \neq 1$ we have that $\Riccal_\GF^\pm=\Riccal_\SV^\pm$.
\end{theorem}

\begin{proof}Let $D$ be a metric connection with divergence $\div$ and pure-type torsion. From \eqref{eq:pure-torsion}, we have that, for $c_\pm\in V_\pm$,
\[
[[c_\pm,a_\mp]_\mp,b_\pm]_\pm=D_{[c_\pm,a_\mp]_\mp}b_\pm.
\]
On the other hand,
\[
\div([a_\mp,b_\pm]_\pm)=\div_D(D_{a_\mp}b_\pm)=\tr(D (D_{a_\mp}b_\pm))=\sum_i \prodesc{\tilde e_i^\pm}{D_{e_i^\pm}D_{a_\mp}b_\pm},
\]
(the terms from $V_\mp$ vanish because $D$ is compatible with $\Gcal$), and
\[
\Lcal_{\rho a_\mp}(\div b_\pm)=\Lcal_{\rho a_\mp}(\tr(D b_\pm))=\sum_i (\prodesc{D_{a_\mp}\tilde e_i^\pm}{D_{e_i^\pm}b_\pm}+\prodesc{\tilde e_i^\pm}{D_{a_\mp}D_{e_i^\pm}b_\pm}).
\]
Therefore, using that $[e_i^\pm,a_\mp]=-[a_\mp,e_i^\pm]$ and that $D$ preserves traces:
\begin{align*}
\Riccal^\pm_\SV(a_\mp,b_\pm) &= \sum_i \big( \prodesc{\tilde e_i^\pm}{D_{e_i^\pm}D_{a_\mp}b_\pm-D_{a_\mp}D_{e_i^\pm}b_\pm-D_{[e_i^\pm,a_\mp]_\mp}b_\pm}-\prodesc{D_{a_\mp}\tilde e_i^\pm}{D_{e_i^\pm}b_\pm} \big) \\
&=\sum_i \big( \prodesc{\tilde e_i^\pm}{\Rcal_\GF^\pm(e_i^\pm,a_\mp)b_\pm +D_{[e_i^\pm,a_\mp]_\pm}b_\pm}-\prodesc{D_{a_\mp}\tilde e_i^\pm}{D_{e_i^\pm}b_\pm} \big) \\
&= \Riccal^\pm_\GF(a_\mp,b_\pm)+\sum_i \big( \prodesc{\tilde e_i^\pm}{D_{[e_i^\pm,a_\mp]_\pm}b_\pm}-\prodesc{D_{a_\mp}\tilde e_i^\pm}{D_{e_i^\pm}b_\pm} \big) \\
&=\Riccal^\pm_\GF(a_\mp,b_\pm)-\sum_i \big( \prodesc{\tilde e_i^\pm}{D_{D_{a_\mp}e_i^\pm}b_\pm}+\prodesc{D_{a_\mp}\tilde e_i^\pm}{D_{e_i^\pm}b_\pm} \big) \\
&=\Riccal^\pm_\GF(a_\mp,b_\pm)-\sum_i \big( \prodesc{(D b_\pm)^*\tilde e_i^\pm}{D_{a_\mp}e_i^\pm}+\prodesc{(D b_\pm)^*(D_{a_\mp}\tilde e_i^\pm)}{e_i^\pm} \big) \\
&=\Riccal^\pm_\GF(a_\mp,b_\pm)-\sum_i \big( \Lcal_{\rho a_\mp}\prodesc{(D b_\pm)^* \tilde e_i^\pm}{e_i^\pm}- \prodesc{D_{a_\mp}((D b_\pm)^*\tilde e_i^\pm)}{e_i^\pm} \\
& \qquad \qquad +\prodesc{(D b_\pm)^*(D_{a_\mp}\tilde e_i^\pm)}{e_i^\pm} \big) \\
&=\Riccal^\pm_\GF(a_\mp,b_\pm)-\sum_i \big( \Lcal_{\rho a_\mp}\prodesc{(D b_\pm)^* \tilde e_i^\pm}{e_i^\pm}- \prodesc{(D_{a_\mp}(D b_\pm)^*)\tilde e_i^\pm}{e_i^\pm} \big)\\
&=\Riccal^\pm_\GF(a_\mp,b_\pm)-\Lcal_{\rho a_\mp}(\tr((D b_\pm)^*))+\tr(D_{a_\mp}(D b_\pm)^*) \\
&=\Riccal^\pm_\GF(a_\mp,b_\pm)-(D_{a_\mp}\tr)((D b_\pm)^*) \\
&=\Riccal^\pm_\GF(a_\mp,b_\pm). \qedhere
\end{align*}
\end{proof}

As a corollary of Theorems \ref{thm:ricci} and \ref{thm:ricci-2}, we also have that 
\[ \Riccal_\SSCV(a_\mp,b_\pm)=\Riccal^\pm_\SV(a_\mp,b_\pm)+\Riccal^\mp_\SV(b_\pm,a_\mp).\]
It was pointed out to us by the referee that this identity is  proved in \cite[Prop. 4.5]{streets-strickland-constable-valach}. In this sense, Theorem \ref{thm:ricci-2} can be considered as a refinement of it.

\section{Total curvatures}

One of the advantages of $\Rcal_\GF$ over $\Rcal_\JV$ is that it can be used to define tensorial curvatures for $E$-connections on vector bundles. An $E$-connection on a vector bundle $W\to M$ is a linear map $D:\Gamma(W)\to\Gamma(E^*\otimes W)$ satisfying the Leibniz rule
\[
D(fw)=\rho^*df\otimes w+f Dw, \text{for $f\in\Cinf(M)$ and $w\in\Gamma(W)$.}
\]
If $E$ carries a generalized metric $\Gcal$, then we can define the curvatures of $D$ as the tensors $\Rcal^\pm\in\Gamma(V_\pm^*\otimes V^*_\mp\otimes\End W)$ by restricting the naïve curvature to the corresponding subbundles, and we can assemble these into a single total curvature.

\begin{definition}
The \textbf{total curvature} of $D$ is the tensor $\Rcal\in\Gamma(E^*\otimes E^*\otimes\End W)$ given by $\Rcal:=\Rcal^++\Rcal^-$. Explicitly,
\[
\Rcal(a,b)=\Rcal^+(a_+,b_-)+\Rcal^-(a_-,b_+).
\]
\end{definition}

Notice that because $\Rcal^+(a_+,b_-)=-\Rcal^-(b_-,a_+)$, we have that actually
\[
\Rcal\in\Gamma(\wedge^2E^*\otimes\End W).
\]

In the case of an $E$-connection on $E$, taking the trace of the total curvature gives the total Ricci curvature.

\begin{definition}\label{def:total-Ricci}
Let $D$ be a generalized connection on $E$ with total curvature $\Rcal$ for some generalized metric $\Gcal$. The \textbf{total generalized Ricci curvature} is the tensor $\Riccal\in\Gamma(E^*\otimes E^*)$ given by
\[
\Riccal(a,b):=\tr(\Rcal(\cdot,a)b).
\]
\end{definition}

In the case that $D$ is compatible with $\Gcal$, it is actually given by a familiar formula.

\begin{lemma} \label{lem:total-ricci}
If $D$ is a generalized connection on $E$ compatible with $\Gcal$, then 
\[
\Riccal=\Riccal_\GF^++\Riccal_\GF^-.
\]
\end{lemma}

\begin{proof}
% Straightforward computation, 
By using the compatibility of $D$ with $\Gcal$:
\begin{align*}
\Riccal(a,b) &=\sum_i\prodesc{\Rcal(e_i^+,a)b}{\tilde e_i^+}+\sum_j\prodesc{\Rcal(e_j^-,a)b}{\tilde e_j^-} \\
&=\sum_i\prodesc{\Rcal_\GF^+(e_i^+,a_-)b_+}{\tilde e_i^+}+\sum_j\prodesc{\Rcal_\GF^-(e_j^-,a_+)b_-}{\tilde e_j^-} \\
&= \Riccal_\GF^+(a_-,b_+)+\Riccal_\GF^-(a_+,b_-). \qedhere
\end{align*}
\end{proof}

In \cite{garcia-fernandez-spinors} the alternative tensor $\Riccal'_\GF:=\Riccal^+_\GF-\Riccal^-_\GF$ was considered, and the question of when it is skew-symmetric was posed, because of its relevance to the study of the generalized Ricci flow. Notice that $\Riccal'_\GF(a,b)=\Riccal(a,\Gcal b)$, so we can translate the question of whether $\Riccal'_\GF$ is skew-symmetric to $\Riccal$ being symmetric. %, as the following shows.

\begin{lemma}\label{lem:Ric-sym-skew}
For $(\Gcal,\div)$ on $E$, we have that $\Riccal(a,\Gcal b)=-\Riccal(\Gcal a,b)$. In particular, $\Riccal'_\GF$ is skew-symmetric if and only if $\Riccal$ is symmetric.
\end{lemma}

\begin{proof}
By a straightforward computation:
\[
\begin{aligned}
\Riccal(a,\Gcal b) &= \Riccal^+_\GF(a_-,b_+)-\Riccal^-_\GF(a_+,b_-) \\
&= -\Riccal^+_\GF(\Gcal a_-, b_+)-\Riccal^-_\GF(\Gcal a_+,b_-) \\
&=-\Riccal(\Gcal a,b).
\end{aligned}
\]
If $\Riccal'_\GF$ is skew-symmetric, then
\[
\Riccal(a,b)=\Riccal'_\GF(a,\Gcal b)=-\Riccal'_\GF(\Gcal b,a)=-\Riccal(\Gcal b,\Gcal a)=\Riccal(b,a).
\]
The converse is analogous.
\end{proof}

From our perspective, the tensor $\Riccal$ seems like a natural choice than $\Riccal'_\GF$ for the study of generalized Ricci flow, as we will now explain. In \cite{garcia-fernandez-spinors}, generalized Ricci flow is presented as an equation for a family of pairs $(\Gcal_t,\div_t)$. Observe that if $\Gcal_t$ is a generalized metric for each $t$, then $\partial_t\Gcal_t$ is a symmetric endomorphism that interchanges $V_+$ and $V_-$. Indeed, since $\Gcal_t^*=\Gcal_t$, we have that $(\partial_t\Gcal_t)^*=\partial_t\Gcal_t$, and so, if $\Gcal_ta_\pm=\pm a_\pm$, and the same for $b_\pm$,
\[
0=\der{}{t}\prodesc{\Gcal_t a_\pm}{\Gcal_t b_\pm}=\pm\prodesc{\partial_t\Gcal_t a_\pm}{b_\pm}\pm\prodesc{a_\pm}{\partial_t\Gcal_t b_\pm}=\pm 2\prodesc{\partial_t\Gcal_t a_\pm}{b_\pm},
\]
i.e., $\partial_t\Gcal_t:V_\pm\to V_\mp$. Let $\partial_t\Gcal_t^\pm$ be the piece corresponding to $V_\mp$, then the generalized Ricci flow in \cite{garcia-fernandez-spinors} is the equation
\begin{equation} \label{eq:generalized-ricci-flow}
\partial_t\Gcal_t^+=-2\Riccal^+_\GF(\Gcal_t,\div_t),
\end{equation}
as endomorphisms $V_-\to V_+$. One can of course consider the analogous equation for the $V_+$ piece. Furthermore, one could assemble these two into a single equation
\begin{equation} \label{eq:total-generalized-ricci-flow}
\partial_t\Gcal_t=-2\Riccal(\Gcal_t,\div_t).
\end{equation}
Observe that this equation only makes sense if $\Riccal$ is symmetric. Indeed, the tangent space to the space of generalized metrics at $\Gcal_t$ is composed of symmetric endomorphisms $\dot\Gcal_t$ that interchange $V_+$ and $V_-$ and such that $\Gcal_t\dot\Gcal_t+\dot\Gcal_t\Gcal_t=0$. The tensor $\Riccal$ satisfies this last property by Lemma \ref{lem:Ric-sym-skew}, so that it defines a vector field on the space of generalized metrics if and only if it is symmetric. In this case, \eqref{eq:total-generalized-ricci-flow} is well defined and equivalent to \eqref{eq:generalized-ricci-flow}. 

Similar considerations were made in \cite{garcia-fernandez-spinors} using $\Riccal'_\GF$. The difference between the use of $\Riccal$ or $\Riccal'_\GF$ is one of convention or perspective: a generalized metric $\Gcal$ can be regarded as a new product $(\cdot,\cdot):=\prodesc{\Gcal\cdot}{\cdot}$ on $E$, and the variation of the family of products $(\cdot,\cdot)_t$ can be regarded as an endomorphism of $E$ using $\prodesc{\cdot}{\cdot}$ or $(\cdot,\cdot)_t$. The first approach leads to $\partial_t\Gcal_t$ and to generalized Ricci flow in the form \eqref{eq:total-generalized-ricci-flow}, while the second one leads to $\Gcal_t^{-1}\partial_t\Gcal_t$ and to generalized Ricci flow as in \cite{garcia-fernandez-spinors}.

%We find \eqref{eq:total-generalized-ricci-flow} more natural by its analogy with the classical case and hence focus on the question of whether $\Riccal$ is symmetric. 

% and the question of when $\Riccal'_\GF$ is skew-symmetric was posed.

% We gave an answer to this question, and after sharing our work we were told that it had been proven independently, to appear in \cite{garcia-fernandez-gonzalez-molina-streets}, and that actually an answer to the question was implicit in the proof of \cite[Prop. 3.2]{severa-valach}. Nevertheless, we add the proof here for completeness.

We finish by giving an answer to the question of when $\Riccal$ is symmetric. When we shared our work, we were told that it has been proven independently (as it has later appeared in \cite[Lem. 2.7]{garcia-fernandez-gonzalez-molina-streets}) and also that, given Theorem \ref{thm:ricci-2} above, the answer can be recovered from the proof of \cite[Prop. 3.2]{severa-valach}.

%% STILL TO FIX

% We apply Theorems \ref{thm:ricci} and \ref{thm:ricci-2} to answer the question of when the tensor
% \begin{equation}\label{eq:total-GF}
% \Riccal'_\GF:=\Riccal^+_\GF-\Riccal^-_\GF\in\Gamma(E^*\otimes E^*)
% \end{equation}
% is skew-symmetric. This was posed in  \cite{garcia-fernandez-spinors}  because of its relevance to the study of the generalized Ricci flow. When we shared our work, we were told that Proposition \ref{prop:symmetric-ricci} below
% has been proved independently, as will  appear in \cite{garcia-fernandez-gonzalez-molina-streets}. 

% We modify slightly the definition in order to focus on the symmetry of the Ricci tensor, which from the perspective of this work seems a more natural condition.

% \begin{definition}\label{def:total-Ricci}
% For $(\Gcal,\div)$ on $E$, the \textbf{total generalized Ricci curvature} is the tensor $\Riccal_\GF:=\Riccal^+_\GF+\Riccal_\GF^-\in\Gamma(E^*\otimes E^*)$, that is,
% \[
% \Riccal_\GF(a,b)=\Riccal^+_\GF(a_-,b_+)+\Riccal^-_\GF(a_+,b_-). 
% \] % \quad\text{ for $a,b\in\Gamma(E)$
% \end{definition}

% Note that  $\Riccal'_\GF(a,b)=\Riccal_\GF(a,\Gcal b)$ and, by Theorem \ref{thm:ricci}, we have that $\Riccal_\SSCV$ is the symmetrization of $\Riccal_\GF$. We first show the relation between \eqref{eq:total-GF} and   Definition~\ref{def:total-Ricci}.

A key ingredient is the following definition.

\begin{definition}[\cite{severa-valach}]
A pair $(\Gcal,\div)$ on $E$ is said to be \textbf{compatible} if
\[
\div([a_\mp,b_\pm])-\Lcal_{\rho a_\mp}(\div b_\pm)+\Lcal_{\rho b_\pm}(\div a_\mp)=0,\quad\text{ for $a_\pm,b_\pm\in\Gamma(V_\pm)$.} \qedhere
\]
\end{definition} 

This condition can be phrased as $\Gcal$ being invariant under the 1-parameter family of Courant algebroid automorphisms generated by $\div$ \cite{streets-strickland-constable-valach}. Finally:

\begin{theorem} \label{prop:symmetric-ricci}
For a pair $(\Gcal,\div)$ on $E$ we have that $\Riccal$ is symmetric if an only if $(\Gcal,\div)$ is compatible.
\end{theorem}

\begin{proof}
By Lemma \ref{lem:total-ricci}, $\Riccal$ is symmetric in and only if
\[
\Riccal^\pm_\GF(a_\mp,b_\pm)=\Riccal^\mp_\GF(b_\pm,a_\mp).
\]
By Theorem \ref{thm:ricci-2}, the left-hand side is
\begin{equation} \label{eq:sv-left}
\div([a_\mp,b_\pm]_\pm)-\Lcal_{\rho a_\mp}(\div b_\pm)-\tr_\pm [[\cdot,a_\mp]_\mp,b_\pm]_\pm,
\end{equation}
and the right-hand side,
\begin{equation} \label{eq:sv-right}
\div([b_\pm,a_\mp]_\mp)-\Lcal_{\rho b_\pm}(\div a_\mp)-\tr_\mp [[\cdot,b_\pm]_\pm,a_\mp]_\mp.
\end{equation}
Note that $[\cdot,a_\pm]_\pm$ defines a morphism $V_\mp\to V_\pm$, so that
\begin{align*}
[[\cdot,a_\mp]_\mp,b_\pm]_\pm: & V_\pm\to V_\pm, \\
[[\cdot,b_\pm]_\pm,a_\mp]_\mp: & V_\mp\to V_\mp,
\end{align*}
and by the cyclic property of the trace we have that
\[
\tr_\pm [[\cdot,a_\mp]_\mp,b_\pm]_\pm=\tr_\mp [[\cdot,b_\pm]_\pm,a_\mp]_\mp.
\]
Hence, $\eqref{eq:sv-left}=\eqref{eq:sv-right}$ if and only if
\[
\div([a_\mp,b_\pm])-\Lcal_{\rho a_\mp}(\div b_\pm)+\Lcal_{\rho b_\pm}(\div a_\mp)=0,
\]
i.e., if and only if $(\Gcal,\div)$ is compatible.
\end{proof}

\bigskip
\noindent \textbf{Acknowledgements: } We are grateful to Mario García-Fernández for pointing out the question that led to this work and insightful discussions on preliminary versions. We also thank Filip Mou\v{c}ka, Pavol \v{S}evera, Fridrich Valach and the referee for helpful comments. %, and Pavol \v{S}evera for pointing out \cite[Prop. 3.2]{severa-valach}.

% These different notions of curvature lead to different notions of generalized Ricci curvature, and the aim of this note is to settle what the relations among these are.

%%

\bibliographystyle{alpha}
\bibliography{Ricci-biblio}

\end{document}